\documentclass[a4paper,12pt]{article}

\usepackage{amsmath}
\usepackage{amssymb}
\usepackage{amsthm}

\usepackage{color}
\newcommand{\comment}[1]{{\bf\color{red} #1}}

\newcommand{\R}{\mathbb{R}}
\newcommand{\C}{\mathbb{C}}
\newcommand{\N}{\mathbb{N}}
\newcommand{\RN}{{\mathbb{R}^N}}
\newcommand{\loc}{\mathrm{loc}}
\newcommand{\rad}{\mathrm{rad}}
\newcommand{\supp}{\mathop{\mathrm{supp}}}

\renewcommand{\liminf}{\varliminf}
\renewcommand{\limsup}{\varlimsup}

\makeatletter
\@addtoreset{equation}{section}%
\renewcommand{\theequation}{\thesection.\@arabic\c@equation}
\makeatother

\newtheorem{theorem}{Theorem}[section]
\newtheorem{lemma}[theorem]{Lemma}

\newtheorem{proposition}[theorem]{Proposition}

\newtheorem{claim}{Claim}

\newtheorem*{claim*}{Claim}

\begin{document}
\title{A new rearrangement inequality and its application for 
$L^2$-constraint minimizing problems}
\author{Masataka Shibata}
\maketitle

\section{Introduction}
\label{sec:1}

In this paper, we show a new rearrangement inequality and give some applications to $L^2$-constraint minimizing problems.
In order to explain, we consider the following variational problem.
\begin{align*}
 E_\alpha &= \inf_{u \in M_\alpha} I(u), \\
 I(u) &= \frac{1}{2} \int_{\RN} |\nabla u|^2 dx - \frac{1}{p+1} \int_{\RN} |u|^{p+1} dx, \\
 M_\alpha & = \left\{u \in H^1(\RN); \|u\|_{L^2(\RN)}^2 = \alpha \right\},
\end{align*}
where $\alpha >0$ is a given constant and $N \geq 1$.
In this problem, it is well-known that $E_\alpha > -\infty$ if $1 < p < 1 + 4/N$, and
we can expect the existence of a global minimizer.

Here, we recall the Schwartz rearrangement.
For $u \in H^1(\RN)$, we denote by $u^*$ the Schwartz rearrangement of $u$.
It is well known that $u$ and $u^*$ are equimeasurable, 
$\|u\|_{L^{r}(\RN)} = \|u^*\|_{L^{r}(\RN)}$ for $r \geq 1$, and 
\begin{equation}
\label{eq:17}
 \int_{\RN} |\nabla u^*|^2 dx \leq
 \int_{\RN} |\nabla u|^2 dx.
\end{equation}

Thus $\{u_n^*\}_{n \in \N} \subset M_\alpha$ is a minimizing sequence for any minimizing sequence
$\{u_n\}_{n \in \N} \subset M_\alpha$.
Therefore we can use compactness of the embedding $H^1_{\rad}(\RN) \subset L^{2+4/N}(\RN)$ to obtain a minimizer $u \in M_\alpha$.

%また、$u \in M_\alpha$がminimizerであれば、$I(u^*) = I(u)= E_\alpha$となっているから、
%不等式\eqref{eq:17}は等号で成立していることがわかる。
%このような場合も調べられており、これこれといった$u$に関する情報も得られる\footnote{具体的に}。

In addition, precompactness of any given minimizing sequence is important.
Let $u$ be a global minimizer then $u$ is a solution of 
\begin{equation*}
 -\Delta u + \mu u = |u|^{p-1} u \text{ in } \RN,
\end{equation*}
where $\mu$ is a Lagrange multiplier.
Put $v(t,x)= e^{i \mu t} u(x)$ then $v$ is a standing wave of the following nonlinear Schr\"{o}dinger equation.
\begin{equation*}
 i v_t = \Delta v + |v|^{p-1} v.
\end{equation*}
In \cite{MR677997}, by using $H^1$-precompactness of any minimizing sequences, they showed 
orbital stability of the set of global minimizers.
For this purpose, the subadditivity condition
\begin{equation}
\label{eq:18}
E_{\alpha+\beta}  < E_\alpha + E_\beta
\end{equation}
plays an important rule.
The subadditivity condition exclude the dichotomy of minimizing sequences, and it implies $H^1$-precompactness.
In addition, the scaling arguments has been used to show the subadditivity condition.
In this paper, we give an another proof to obtain the subadditivity condition.
Let $u \in M_\alpha$ and $v \in M_\beta$ be a minimizer of $E_\alpha$ and  $E_\beta$.
We construct $w$ satisfying the following inequality.
\begin{equation}
\label{eq:19}
\|w\|_{L^r}^r = \|u\|_{L^r}^r + \|v\|_{L^r}^r, \quad
\int_{\RN} |\nabla w|^2 dx <
\int_{\RN} |\nabla u|^2 dx +
\int_{\RN} |\nabla v|^2 dx,
\end{equation}
where $r \geq 1$.
Therefore $w \in M_{\alpha+ \beta}$ and 
\begin{equation*}
E_{\alpha + \beta}  \leq I(w) < I(u) + I(v) = E_\alpha + E_\beta.
\end{equation*}
Hence \eqref{eq:18} holds.
Our main result is to construct such $w$ by using a new rearrangement.
Since it does not require scaling arguments, we can apply $L^2$-constraint minimizing problem related to nonlinear elliptic systems.

This paper is organized as follows.
In Section \ref{sec:2}, we introduce a new rearrangement and state our main theorem.
In Section \ref{sec:3}, we state application to the subadditivity condition.
In Section \ref{sec:4}, we state application to nonlinear elliptic systems.

\section{Rearrangement}
\label{sec:2}

In this section, we introduce a new rearrangement and show our main results.
For the purpose, we recall the Steiner rearrangement.

\subsection{The Steiner rearrangement}

In the following, we write $x=(x_1, x')$ with $x_1 \in \R$, $x' \in \R^{N-1}$
and we denote by $\mathcal{L}^i$ the $i$-dimensional Lebesgue measure.
Let $u$ be a function satisfies the following condition (A).
\begin{itemize}
 \item[(A)] $u : \R^N \to \R$: measurable, 
$\lim_{|x| \to \infty} u(x)=0$.
\end{itemize}
We denote by $u^\star$ the Steiner symmetric rearrangement of $u$. 
The Steiner symmetric rearrangement $u^\star$ is a function which satisfies the following properties:
\begin{itemize}
 \item $x_1 \mapsto u(x_1, x')$ is symmetric with respect to the origin and non-increasing with respect to $|x_1|$ for any $x' \in \R^{N-1}$. 
 \item $u^\star(\cdot, x')$ is equimeasurable with $u(\cdot, x')$ for any $x' \in \R^{N-1}$. 
That is, 
for any $t>0$, $x' \in \R^{N-1}$,
\begin{equation}
\label{eq:1}
\mathcal{L}^1\left(
\left\{x_1 \in \R; |u(x_1,x')|>t\right\}
\right)
=
\mathcal{L}^1\left(
\left\{x_1 \in \R; u^\star(x_1,x')>t\right\}
\right).
\end{equation}
\end{itemize}
More precisely,
the Steiner rearrangement $u^\star$ is defined by
\begin{equation*}
 u^\star(x_1, x')=
\int_0^\infty \chi_{\{|u(\cdot, x')|>t\}^\star}(x_1) dt,
\end{equation*}
where
$A^\star$ is the Steiner rearrangement of $A$ defined by
\begin{equation*}
A^\star = \left(
-\mathcal{L}^1(A)/2,
\mathcal{L}^1(A)/2
\right).
\end{equation*}
We remark that the Steiner rearrangement is defined under more general assumptions.
However, for simplicity, we assume the condition (A).
About the Steiner rearrangement, we summarize well-known facts as follows.
%(see, e.g., \footnote{何かを参照}).
\begin{proposition}
\label{prop:1}
Assume $u$ satisfies (A) and let $u^\star$ be the Steiner rearrangement of $u$.
Then
\begin{enumerate}
 \item $u^\star$ is measurable in $\RN$. Moreover, 
$|u|$ and $u^\star$ is equimeasurable in $\RN$, that is, 
\begin{equation*}
\mathcal{L}^N\left(
\left\{x \in \R^N; |u(x)|>t\right\}
\right)
=
\mathcal{L}^N\left(
\left\{x \in \R^N; u^\star(x)>t\right\}
\right).
\end{equation*}
 \item  Let $\Phi_1, \Phi_2: [0,\infty) \to \R$ be monotone functions.
For $\Phi=\Phi_1 + \Phi_2$,
\begin{equation*}
\int_{\RN} \Phi(u^\star) dx =
\int_{\RN} \Phi(|u|) dx 
\end{equation*}
holds if
\begin{equation*}
\left| \int_{\RN} \Phi_1(|u|) dx \right| < \infty \text{ or }
\left| \int_{\RN} \Phi_2(|u|) dx \right| < \infty.
\end{equation*}
In particular, 
\begin{equation*}
 \int_{\RN} |u^\star|^p dx
=
 \int_{\RN} |u|^p dx
\end{equation*}
for $1 \leq p < \infty$.
 \item Assume $1 \leq p < \infty$.
If $u \in W^{1,p}(\RN)$,
it holds that $u^\star \in W^{1,p}(\RN)$. Moreover, 
\begin{equation*}
\int_{\RN} |\partial_i u^\star|^p dx
\leq
\int_{\RN} |\partial_i u|^p dx \text{ for } i=1, \dots, N.
\end{equation*}
\end{enumerate}
\end{proposition}

\subsection{Coupled rearrangement}

Now we introduce a new rearrangement which we call \textit{coupled rearrangement}.
Suppose $u$ and $v$ satisfy the condition (A).
The coupled rearrangement $u \star v$ of $u$ and $v$ is defined as follows.
For any $x' \in \R^{N-1}$, 
$x_1 \mapsto (u \star v)(x_1, x')$ is symmetric with respect to the origin and monotone with respect to $|x_1|$.
For any $t>0$, $x' \in \R^{N-1}$,
\begin{align}
\notag
&\mathcal{L}^1\left(
\left\{x_1 \in \R; |u(x_1,x')|>t\right\}
\right)
+
\mathcal{L}^1\left(
\left\{x_1 \in \R; |v(x_1,x')|>t\right\}
\right) \\
&=
\mathcal{L}^1\left(
\left\{x_1 \in \R; (u \star v)(x_1,x')>t\right\}
\right).
\end{align}
More precisely,
$u \star v$ is defined by
\begin{equation*}
(u \star v)(x_1, x')=
\int_0^\infty \chi_{\{|u(\cdot, x')|>t\} \star \{|v(\cdot, x')|>t\}}(x_1) dt,
\end{equation*}
where 
\begin{equation*}
A \star B
= \left(
-(\mathcal{L}^1(A) + \mathcal{L}^1(B))/2,
(\mathcal{L}^1(A) + \mathcal{L}^1(B))/2
\right).
\end{equation*}
About the coupled rearrangement, we can show similar properties as follows.
We give the proofs in the next subsection.
\begin{lemma}
\label{lem:3}
Assume $u$ and $v$ satisfy the condition (A) and let $u \star v$ be the coupled rearrangement of $u$ and $v$.
Then, 
\begin{enumerate}
 \item 
$u \star v$ is measurable in $\R^N$. 
Moreover, 
\begin{align*}
&\mathcal{L}^N\left(
\left\{x \in \R^N; |u(x)|>t\right\}
\right)
+
\mathcal{L}^N\left(
\left\{x \in \R^N; |v(x)|>t\right\}
\right) \\
&=
\mathcal{L}^N\left(
\left\{x \in \R^N; (u \star v)(x)>t\right\}
\right).
\end{align*}
 \item 
Let $\Phi_1, \Phi_2: [0,\infty) \to \R$ be monotone functions.
For $\Phi=\Phi_1 + \Phi_2$,
\begin{equation*}
\int_{\RN} \Phi(u \star v) dx
=
\int_{\RN} \Phi(|u|) dx
+
\int_{\RN} \Phi(|v|) dx,
\end{equation*}
where
\begin{equation*}
\left| \int_{\RN} \Phi_1(|u|) dx \right|,
\left| \int_{\RN} \Phi_1(|v|) dx \right| < \infty
\text{ or }
\left| \int_{\RN} \Phi_2(|u|) dx \right|,
\left| \int_{\RN} \Phi_2(|v|) dx \right| < \infty
\end{equation*}
holds.
In particular, 
\begin{equation*}
\int_{\RN} |u \star v|^p dx
=
\int_{\RN} |u|^p dx
+
\int_{\RN} |v|^p dx
\end{equation*}
holds for any $p \geq 1$.
\end{enumerate}
\end{lemma}
%また、微分についても、次の不等式が成立する。
\begin{lemma}
\label{lem:2}
Assume $1 \leq p < \infty$.
$u$ and $v$ satisfy the condition (A) and
$u,v \in W^{1,p}(\RN)$.
Then it holds that
\begin{equation*}
\int_{\RN} |\partial_i (u \star v)|^p dx
\leq
\int_{\RN} |\partial_i u|^p dx
+
\int_{\RN} |\partial_i v|^p dx \text{ for } i=1,\dots,N.
\end{equation*}
\end{lemma}

Our main theorem is the following strict inequality.
\begin{theorem}
\label{thm:1}
For $u, v \in W^{1,p}(\RN) \cap C^1(\RN)$ satisfying that $u,v >0$, 
$\lim_{|x| \to \infty} u(x) = \lim_{|x| \to \infty} v(x) = 0$, 
and $u(x_1, x')$ and $v(x_1, x')$ are monotone decreasing with respect to $|x_1|$.
Then, the following strict inequality holds.
\begin{equation*}
\int_{\RN} |\nabla (u \star v)|^p dx
<
\int_{\RN} |\nabla u|^p dx
+
\int_{\RN} |\nabla v|^p dx.
\end{equation*}

\end{theorem}

\subsection{Proof of Lemma \ref{lem:3} and \ref{lem:2}}

In this subsection, we give the proofs of Lemma \ref{lem:3} and \ref{lem:2}.
To our purpose, we prepare the following lemma.
\begin{lemma}
\label{lem:1}
Assume that $u$ and $v$ satisfy the condition (A).
Then the following properties holds.
\begin{enumerate}
 \item The Steiner rearrangement and the coupled rearrangement are
       invariant to translation of $x_1$ direction. That is, 
       for $s,t \in \R$, $\tilde{u}^\star=u^\star$ and $\tilde{u} \star
       \tilde{v} = u \star v$ hold, where 
       $\tilde{u}(x_1,x') = u(x_1+s, x')$, 
       $\tilde{v}(x_1,x') = v(x_1+t, x')$.
 \item If $\supp u \cap \supp v = \emptyset$, it holds that $u \star v =(u +
 v)^\star$.
 \item For $s>0$, it holds that
$(|u|-s)_+ \star (|v|-s)_+ = (u \star v -s)_+$.
\end{enumerate}
\end{lemma}

\begin{proof}
(i): It is clear by the definition of the coupled rearrangement.

(ii):
It is sufficient to show
\begin{align}
\notag
& \mathcal{L}^1 \left(
\left\{x_1 \in \R; (u \star v)(x_1, x') >t \right\}
\right) \\
\label{eq:2}
&=
\mathcal{L}^1 \left(
\left\{x_1 \in \R; (u + v)^\star (x_1, x') >t \right\}
\right)
\text{ for  } 
t>0,
x' \in \R^{N-1}.
\end{align} 
Fix $t>0$, $x' \in \R^{N-1}$.
By the definition of the Steiner rearrangement, we have
\begin{equation*}
\mathcal{L}^1 \left(
\left\{x_1 \in \R; (u + v)^\star (x_1, x') >t \right\}
\right)
=
\mathcal{L}^1 \left(
\left\{x_1 \in \R; |u + v| (x_1, x') >t \right\}
\right).
\end{equation*}
Since $\supp u \cap \supp v = \emptyset$, we have
\begin{align*}
& \mathcal{L}^1 \left(
\left\{x_1 \in \R; |u + v| (x_1, x') >t \right\}
\right) \\
&=
\mathcal{L}^1 \left(
\left\{x_1 \in \R; |u| (x_1, x') >t \right\}
\right)
+
\mathcal{L}^1 \left(
\left\{x_1 \in \R; |u| (x_1, x') >t \right\}
\right).
\end{align*}
On the other hand, by the definition of the coupled rearrangement,
we have
\begin{align*}
& \mathcal{L}^1 \left(
\left\{x_1 \in \R; (u \star v)(x_1, x') >t \right\}
\right) \\
&=
\mathcal{L}^1 \left(
\left\{x_1 \in \R; |u|(x_1, x') >t \right\}
\right)
+
\mathcal{L}^1 \left(
\left\{x_1 \in \R; |v|(x_1, x') >t \right\}
\right).
\end{align*}
Consequently, \eqref{eq:2} holds.

(iii):
By the definition of the coupled rearrangement, we can obtain that
\begin{align*}
 &
\mathcal{L}^1 \left(
\left\{x_1 \in \R; ((|u|-s)_+ \star (|v|-s)_+)(x_1, x') >t \right\}
\right) \\
&=
\mathcal{L}^1 \left(
\left\{x_1 \in \R; (|u|-s)_+ (x_1, x') >t \right\}
\right)
+
\mathcal{L}^1 \left(
\left\{x_1 \in \R; (|v|-s)_+(x_1, x') >t \right\}
\right) \\
&=
\mathcal{L}^1 \left(
\left\{x_1 \in \R; |u|(x_1, x') >s+t \right\}
\right)
+
\mathcal{L}^1 \left(
\left\{x_1 \in \R; |v|(x_1, x') >s+t \right\}
\right) \\
&=
\mathcal{L}^1 \left(
\left\{x_1 \in \R; (u \star v)(x_1, x') >s+t \right\}
\right) \\
&=
\mathcal{L}^1 \left(
\left\{x_1 \in \R; (u \star v -s)_+(x_1, x') >s+t \right\}
\right).
\end{align*}

\end{proof}

\begin{proof}[Proof of Lemma \ref{lem:2}]
Fix $s>0$. Since
$\lim_{|x| \to \infty} u(x) = \lim_{|x| \to \infty} v(x) =0$, 
there exists a positive constant $R=R(s)$ such that
\begin{equation*}
 \left\{x; |u(x)|>s \right\}, \left\{x; |v(x)| >s \right\} \subset B(0,R).
\end{equation*}
Putting $x_s=(3R,0)$ and $v_s(x_1,x')=v(x_1-3R, x')$, we have
\begin{equation*}
\left\{x; |v_s(x)| > s\right\} \subset B(x_s, R).
\end{equation*}
Especially, we obtain
\begin{equation}
\label{eq:10}
 \supp (|u|-s)_+ \cap \supp (|v_s|-s)_+ = \emptyset.
\end{equation}
By Lemma \ref{lem:1} (ii),
\begin{equation*}
 (|u|-s)_+ \star (|v_s|-s)_+ = \left\{(|u|-s)_+ + (|v_s|-s)_+ \right\}^\star.
\end{equation*}
By using Proposition \ref{prop:1} (iii) and \eqref{eq:10}, 
\begin{align*}
\int_{\RN} |\nabla \{(|u|-s)_+ \star (|v_s|-s)_+\}|^p  dx 
&=
\int_{\RN} |\nabla \{(|u|-s)_+ + (|v_s|-s)_+\}^\star|^p  dx \\
&\leq
\int_{\RN} |\nabla \{(|u|-s)_+ + (|v_s|-s)_+\}|^p \\
&=
\int_{\RN} |\nabla (|u|-s)_+|^p 
+
\int_{\RN} |\nabla (|v_s|-s)_+|^p \\
&=
\int_{\{x; u(x)>s\}} |\nabla u|^p 
+
\int_{\{x; v_s(x)>s\}} |\nabla v_s|^p \\
&\leq
\int_{\RN} |\nabla u|^p 
+
\int_{\RN} |\nabla v|^p.
\end{align*}
By Lemma \ref{lem:1} (iii)
\begin{equation*}
(|u|-s)_+ \star (|v|-s)_+
=
(u \star v -s)_+.
\end{equation*}
Thus we get 
\begin{equation*}
(u \star v - s)_+ 
= 
\left\{
(|u|-s)_+ + (|v|-s)_+
\right\}^\star.
\end{equation*}
Therefore, $(u \star v -s)_+$ is Lebesgue measurable for any $s>0$.
It means that $u \star v$ is Lebesgue measurable.

By Lemma \ref{lem:1} (i),
\begin{equation*}
(u-s)_+ \star (v_s-s)_+ = (u-s)_+ \star (v-s)_+.
\end{equation*}
Thus
\begin{align*}
\int_{\RN} |\nabla (u \star v -s)_+|^p dx
&=
\int_{\{x; (u \star v)(x)>s\}} |\nabla (u \star v)|^p dx \\
&=
\int_{\RN} |\nabla (u \star v)|^p \chi_{\{x; (u \star v)(x)>s\}} dx.
\end{align*}
Since $\{x; (u \star v)(x)>s\}$ converges to $\RN$ monotonically as $s \to 0$, 
we can apply the monotone convergence theorem to obtain
\begin{equation*}
\lim_{s \to 0}  
\int_{\RN} |\nabla (u \star v)|^p \chi_{\{x; (u \star v)(x)>s\}} dx
=
\int_{\RN} |\nabla (u \star v)|^p dx.
\end{equation*}
It means the conclusion. 
\end{proof}

\subsection{Proof of Theorem \ref{thm:1}}

To prove Theorem \ref{thm:1}, the next lemma is essential. 
\begin{lemma}
\label{lem:7}
Assume $f,g \in C^1(\R, \R)$, $f,g > 0$, 
$\lim_{|x| \to \infty} f(x) = \lim_{|x| \to \infty} g(x)=0$, and
$f$ and $g$ are non-increasing with respect to $|x|$.
Then the strict inequality
\begin{equation}
\label{eq:9}
\int_{\R} |(f \star g)'|^p dx
<
\int_{\R} |f'|^p dx
+
\int_{\R} |g'|^p dx
\end{equation} 
holds for $1 \leq  p< \infty$.
\end{lemma}
The key ingredient of the proof of Lemma \ref{lem:7} is the quantitative version of the decreasing rearrangement inequality. 
Here we recall the decreasing rearrangement as follows.
Let $f \in PC^1([0,b])$ and let
\begin{equation*}
\mu(\lambda) = \mathcal{L^1} \left\{ x \in [0,b]; f(x) > \lambda \right\}, \quad \lambda \in \R,
\end{equation*}
where $PC^1([0,b])$ is the set of piecewise $C^1$ functions.
$f^\#(x) = \mu^{-1}(x)$ $(x \in [0,b])$ is called the decreasing rearrangement of $f$.
$N_f(\lambda)$ is the multiplicity of $f$ at the level $\lambda$, that is, 
\begin{equation*}
 N_f(\lambda) = \# \left\{ y \in [a,b]; f(y)= \lambda\right\},
\end{equation*}
where $\# A$ means the number of elements of the set $A$.
Then we have the following key results.
\begin{theorem}[{\cite[Theorem 1]{MR0257304}}]
\label{thm:2}
Let $f^\#$ be the decreasing rearrangement of $f \in PC^1([0,b])$.
For any $p \geq 1$, The following inequality holds:
\begin{equation*}
\int_0^b |(f^\#)'(x)|^p dx 
\leq 
\int_0^b \left|
\frac{f'(x)}{N_f(f(x))}
\right|^p dx.
\end{equation*}
\end{theorem}
In {\cite[Theorem 1]{MR0257304}}, Duff showed Theorem \ref{thm:2} for $f \in C^1([0,b])$, but his proof can be modified slightly even for $f \in PC^1([0,b])$.

\begin{proof}[Proof of Lemma \ref{lem:7}]
First, we prepare the following claim.
\begin{claim*}
\begin{equation*}
 \int_{-L}^L |(f^\star)'(x)|^p dx \leq 
2^p \int_{-L}^{L} \left| \frac{f'(y)}{N_f(f(y))}\right|^p dy
\end{equation*}
for any $f \in PC^1([-L,L])$ with $f(-L)=f(L)=0$.
\end{claim*}
Put $g(x)=f(x-L)$. Then $g \in PC^1([0,2L])$. Since $f$ and $g$ are equimeasurable, 
by using the definition of rearrangements, we can obtain
\begin{equation*}
f^\star(x) = g^\#(2x) \text{ for } x \in [0,L].
\end{equation*}
Thus we have
\begin{equation*}
 \int_{-L}^L |(f^\star)'(x)|^p dx = 2^p \int_0^{2L} |(g^\#)'(y)|^p dy.
\end{equation*}
Applying Theorem \ref{thm:2}, we get
\begin{equation*}
\int_0^{2L} |(g^\#)'(y)|^p dy
\leq 
\int_{0}^{2L} \left| \frac{g'(y)}{N_g(g(y))}\right|^p dy
=
\int_{-L}^{L} \left| \frac{f'(y)}{N_f(f(y))}\right|^p dy.
\end{equation*}
Therefore, the claim holds.

Next, let $f$ and $g$ satisfy the assumptions the lemma.
For sufficiently small $s>0$, we have that $(f-s)_+ \not \equiv 0$ and 
$(g-s)_+ \not \equiv 0$.
Since each support of $(f-s)_+$ and $(g-s)_+$ is compact, 
there are large $x_0$ and $L$ such that
\begin{align*}
& \supp (f-s)_+ \cap \supp (g(\cdot - x_0)-s)_+ = \emptyset, \\
& h= \supp (f-s)_+ + \supp (g(\cdot - x_0)-s)_+\in PC^1([-L,L]), \\
& h(-L)=h(L).
\end{align*}
Thus, we can apply the above claim to obtain
\begin{equation}
\label{eq:11}
\int_{-L}^L |(h^\star)'(x)|^p dx \leq 
2^p \int_{-L}^{L} \left| \frac{h'(y)}{N_h(h(y))}\right|^p dy
\end{equation}
By Lemma \ref{lem:1} (i) and (iii), we have
\begin{equation}
\label{eq:12}
\int_{\{x; (f \star g)(x) > s\}} |(f \star g)'(x)|^p dx
=
\int_\R |(f \star g - s)_+'(x)|^p dx
=
\int_{-L}^L |(h^\star)'(x)|^p dx.
\end{equation}
On the other hand, since $(f-s)_+ \in PC^1([-L,L])$, $(f-s)_+ \not=0$, and $(f(-L)-s)_+=(f(L)-s)_+=0$,
we have
\begin{equation*}
 N_f(\lambda) \geq 2 \text{ for } \lambda \in \left[0,\max_{\R} f -s\right).
\end{equation*}
Similarly about $g$, we have
\begin{equation*}
 N_g(\lambda) \geq 2 \text{ for } \lambda \in \left[0,\max_{\R} g -s\right).
\end{equation*}
Therefore, we obtain
\begin{align*}
& N_h(\lambda) \geq 2 \text{ for } \lambda \in \left[0,\max \{\max_{\R} f, \max_{\R} g\}  -s\right), \\
& N_h(\lambda) \geq 4 \text{ for } \lambda \in \left[0,\min \{\max_{\R} f, \max_{\R} g\}  -s\right).
\end{align*}
It asserts that%\footnote{layer cake representation}
\begin{equation}
\label{eq:13}
\int_{-L}^{L} \left| \frac{h'(y)}{N_h(h(y))}\right|^p dy
<
\frac{1}{2^p} \int_{-L}^{L} |h'(y)|^p dy.
\end{equation}
By the definition of $g$, it is clear that
\begin{equation}
\label{eq:14}
\int_{\{x; f(x)>s\}} |f'(x)|^p dx
+
\int_{\{x; g(x)>s\}} |g'(x)|^p dx
=
\int_{-L}^L |h'(x)|^p dx.
\end{equation}
Combining \eqref{eq:11}, \eqref{eq:12}, \eqref{eq:13}, and \eqref{eq:14}, 
we get
\begin{equation}
\label{eq:15}
\int_{\{x; (f \star g)(x) > s\}} |(f \star g)'(x)|^p dx
<
\int_{\{x; f(x)>s\}} |f'(x)|^p dx
+
\int_{\{x; g(x)>s\}} |g'(x)|^p dx.
\end{equation}
Moreover, we can apply Lemma \ref{lem:2} for $\min \{f,s\}$ and 
$\min \{g,s\}$ to obtain
\begin{align}
\notag
\int_{\{x; (f \star g)(x) \leq s\}} |(f \star g)'(x)|^p dx
&= \int_{\R} |(\min\{f \star g, s\})'(x))|^p dx \\
\notag
&=
 \int_{\R} |(\min\{f, s\} \star \min\{g,s\})'(x))|^p dx \\
\label{eq:16}
&\leq
 \int_{\{x; f(x) \leq s\}} |f'(x)|^p dx
+
 \int_{\{x; g(x) \leq s\}} |g'(x)|^p dx.
\end{align}
\eqref{eq:15} and \eqref{eq:16} complete the lemma.
\end{proof}

Now, we can prove Theorem \ref{thm:1}.
\begin{proof}[Proof of Theorem \ref{thm:1}]
Let $u$ and $v$ be functions  satisfying that 
$u, v \in W^{1,p}(\RN) \cap C^1(\RN)$,
$u,v >0$, 
$\lim_{|x| \to \infty} u(x) = \lim_{|x| \to \infty} v(x) = 0$, 
and $u(x_1, x'), v(x_1, x')$ are monotone decreasing with respect to $|x_1|$.
By using Lemma \ref{lem:7}, we have
\begin{equation*}
\int_{\R} |\partial_1 (u \star v)(x_1, x')|^p dx_1
<
\int_{\R} |\partial_1 u(x_1, x')|^p dx_1
+
\int_{\R} |\partial_1 v(x_1,x')|^p dx_1
\end{equation*} 
for any $x' \in \R^{N-1}$.
Integrating with respect to $x'$ over $\R^{N-1}$, we get
\begin{equation*}
\int_{\RN} |\partial_1 (u \star v)|^p dx
<
\int_{\RN} |\partial_1 u|^p dx
+
\int_{\RN} |\partial_1 v|^p dx.
\end{equation*} 
On the other hand, By Lemma \ref{lem:2}, we have
\begin{equation*}
\int_{\RN} |\partial_i (u \star v)|^p dx
\leq
\int_{\RN} |\partial_i u|^p dx
+
\int_{\RN} |\partial_i v|^p dx \text{ for } i=2,\dots, N.
\end{equation*} 
Therefore, we obtain the theorem.
\end{proof}

\section{Application: the subadditivity condition}
\label{sec:3}

For given $\alpha >0$, we consider the following $L^2$-constraint minimizing problem.
\begin{align*}
 E_\alpha =& \inf_{u \in M_\alpha} I[u], \\
I[u]=&
\frac{1}{2} \int_{\RN} |\nabla u|^2 dx 
-\int_{\RN} F(u) dx, \\
M_\alpha=&
\left\{
u \in H^1(\RN); \|u\|_{L^2(\RN)}^2=\alpha
\right\},
\end{align*}
where $F$ satisfies the following assumptions.
\begin{itemize}
 \item[(F1)] $f \in C(\C,\C)$, $f(0)=0$. 
 \item[(F2)] 
$f(r) \in \R$ for $r \in\R$,
$f(e^{i \theta }z)=e^{i \theta}f(z)$ for $\theta \in \R$, $z \in \C$, 
and $F(s)=\int_0^s f(\tau) d \tau$.
 \item[(F3)] $\lim_{z \to 0} f(z)/|z|=0$.
 \item[(F4)] $\lim_{|z| \to \infty} f(z)/|z|^{l-1}=0$, where $l=2+4/N$.
% \item[(F5)] $\liminf_{s \to 0+0} F(s)/s^l=\infty$, where $F(s)=\int_0^s f(\tau) d\tau$.
\end{itemize}
Moreover, we assume that the energy $E_\alpha$ is negative, that is, 
\begin{itemize}
 \item[(E1)] $E_\alpha<0$ for $\alpha>0$.
\end{itemize}
We remark that the condition (E1) is satisfied if $\liminf_{s \to 0} F(s)/s^l = \infty$. (See \cite{shibata-manumath}.)
In \cite{shibata-manumath}, $H^1$-precompactness of minimizing sequences was studied under more general conditions. 
In this section, we give an another proof by using Theorem \ref{thm:1}.

Throughout this section, we assume (F1)--(F4) and (E1) always.
About the energy $E_\alpha$, the following conditions holds.
\begin{lemma}[{\cite[Lemma 2.3]{shibata-manumath}}]
\label{lem:6}
 \begin{enumerate}
  \item $E_{\alpha+\beta} \leq E_\alpha + E_\beta$ for any $\alpha, \beta>0$.
  \item $E_\alpha < E_\beta$ if $\alpha > \beta$.
  \item $\alpha \mapsto E_\alpha$ is continuous on
	$[0,\infty)$.
 \end{enumerate}
\end{lemma}

\begin{lemma}
\label{lem:5}
For any $\alpha>0$, there exists a global minimizer $u \in M_\alpha$.
\end{lemma}
By using the Schwartz rearrangement, (E1),  and compactness of embedding $H^1_{\rad}(\RN) \subset L^p(\RN)$, 
we can obtain a global minimizer. We omit the proof of Lemma \ref{lem:5}.

By using Lemma \ref{lem:5} and the coupled rearrangement, 
we can show the subadditivity condition. Thus we get the following Proposition \ref{prop:2}.
\begin{proposition}
\label{prop:2}
Suppose that (F1)--(F4) and (E1). 
Then, the subadditivity condition \eqref{eq:18} holds. 
Moreover, any minimizing sequence $\{u_n\}_{n \in \N} \subset M_\alpha$
with respect to $E_\alpha$ is precompact. That is, 
taking a subsequence if necessary, there exist $u \in M_\alpha$ and a family $\{y_n\}_{n \in \N} \subset \RN$ such that
$\lim_{n \to \infty} u_n(\cdot - y_n) = u$ in $H^1(\RN)$.
In particular, $u$ is a global minimizer.
\end{proposition}

\begin{proof}[Proof of Proposition \ref{prop:2}]
By the results in \cite{MR677997}, %Cazenave-Lions
it is sufficient to show the subadditivity condition \eqref{eq:18}.
For $\alpha, \beta >0$, 
Lemma \ref{lem:5} asserts that there exist global minimizers 
$u$ and $v$ with respect to $E_\alpha$ and $E_\beta$.
By the elliptic regularity theory, 
$u,v \in C^1(\RN)$ satisfy the condition (A).
Thus we can apply Lemma \ref{lem:3} and Theorem \ref{thm:1} to obtain
\begin{equation*}
E_{\|u \star v\|_{L^2(\RN)}^2} \leq 
 I[u \star v] < I[u] + I[v] =E_{\alpha} + E_{\beta}, \quad
\|u \star v\|_{L^2(\RN)}^2 = \alpha + \beta.
\end{equation*}
Hence \eqref{eq:18} holds.
\end{proof}

\section{Application to $L^2$ constraint minimizing problems related to semi linear elliptic systems}
\label{sec:4}

In this section, we consider the following $L^2$-constraint minimizing problem.

\begin{align*}
 E_{\alpha,\beta} &= \inf_{(u,v) \in M_{\alpha,\beta}} J[u,v], \\
 J[u,v] &= \frac{1}{2} \int_{\RN} |\nabla u|^2 + |\nabla v|^2 dx
-
\int_{\RN} G(|u|^2, |v|^2) dx, \\
M_{\alpha, \beta} &=
\left\{
(u,v) \in H^1(\RN) \times H^1(\RN); 
\|u\|_{L^2(\RN)}^2=\alpha,
\|v\|_{L^2(\RN)}^2=\beta
\right\},
\end{align*}
where $\alpha$ and $\beta$ are nonnegative given constants.
We assume the nonlinear term $G(s)=G(s_1, s_2)$ satisfies that
\begin{description}
 \item[(G1)] $G \in C^1([0,\infty)\times[0,\infty), \R)$, $G(0)=0$.
% and $G$ is a nonnegative function.
 \item[(G2)]  
	    $\lim_{|s| \to 0} g_j(s)=0$ $(j=1,2)$, where $g_j(s)=\frac{\partial G}{\partial s_j}(s)$ $(j=1,2)$.
 \item[(G3)] $\lim_{|s| \to \infty} g_j(s)/|s|^{2/N}=0$ $(j=1,2)$.
 \item[(G4)] $g_j$ is nondecreasing, that is, $g_j(s,t) \leq g_j(s+h,t+k)$ for $s,t,h,k \geq 0$ $(j=1,2)$.
 \item[(G5)] There exists $\sigma>0$ such that $G(s_1,0) + G(0,s_2) < G(s_1, s_2)$ for $0< s_1, s_2 \leq \sigma$.
\end{description}
Moreover, we suppose that
\begin{itemize}
 \item[(E2)] $E_{\alpha, 0}, E_{0,\beta} <0$ for any $\alpha, \beta>0$.
\end{itemize}
%\begin{remark}
%% \item[(G4)] $G(s+h,t+k) + G(s,t) \geq G(s+h,t) + G(s,t+k)$ holds for
%%	    $h,k,s,t \geq 0$.
%%\footnote{より強い条件に。} 
%\end{remark}

%\begin{remark}
%\begin{equation*}
% \liminf_{s \to 0} \frac{G(s)}{s^{1+2/N}} =\infty
%\end{equation*}
%ならOK。
%もう少し弱くするなら、
%\begin{equation*}
% \liminf_{s_1 \to 0} \frac{G(s_1,0)}{s_1^{1+2/N}} =\infty, \quad
% \liminf_{s_2 \to 0} \frac{G(0,s_2)}{s_2^{1+2/N}} =\infty
%\end{equation*}
%が成り立てば十分。
%\end{remark}
This type problem was studied in \cite{MR2548703}. 
In \cite{MR2548703}, they proved the existence of global minimizers.
Our goal in this section is to show $H^1$-precompactness of minimizing sequences as follows.
\begin{theorem}
\label{thm:3}
Assume (G1)--(G5), and (E2).
For $\alpha,\beta \geq 0$,
any minimizing sequence $\{(u_n,v_n)\}_{n \in \N} \subset H^1(\RN) \times H^1(\RN)$
with respect to $E_{\alpha,\beta}$ is pre-compact.
That is, taking a subsequence, there exist
$(u,v) \in M_{\alpha,\beta}$ and $\{y_n\}_{n \in \N} \subset \RN$
such that 
\begin{equation*}
u_n(\cdot - y_n) \to u, \quad
v_n(\cdot - y_n) \to v, \quad
\text{ in } H^1(\RN)
\text{ as } n \to \infty.
\end{equation*}
\end{theorem}

To prove Theorem \ref{thm:3}, we prepare the following lemma. We state the proof of the lemma in Appendix.
\begin{lemma}
\label{lem:4}
The energy $E_{\alpha, \beta}$ satisfies that 
\begin{enumerate}
 \item $E_{\alpha+\alpha', \beta + \beta'} \leq E_{\alpha, \beta} +
       E_{\alpha', \beta'}$ for $\alpha, \beta \geq 0$.
 \item $E_{\alpha, \beta}<0$ for $\alpha,\beta \geq 0$, $(\alpha, \beta)\not=(0,0)$.
 \item $(\alpha, \beta) \mapsto E_{\alpha, \beta}$ is continuous on
       $[0,\infty) \times [0,\infty) \setminus \{(0,0)\}$.
\end{enumerate}
\end{lemma}

\begin{proof}[Proof of Theorem \ref{thm:3}]
In the case $\alpha=0$ or $\beta=0$, the results are included in Proposition \ref{prop:2}.
So we consider the case $\alpha, \beta>0$.
Let $\{(u_n, v_n)\}_{n \in \N}$ be a minimizing sequence in $M_{\alpha,\beta}$.
By using the Gagliardo-Nirenberg inequality,
we have that $\{(u_n, v_n)\}_{n \in \N}$ is bounded in $H^1(\RN)$. (For details, see Lemma \ref{lem:9})
\begin{claim*}
Both $\{u_n\}_{n \in \N}$ and $\{v_n\}_{n \in \N}$ does not vanish, that is, 
\begin{equation*}
\liminf_{n \to \infty}
\left(
\sup_{y \in \RN} \int_{B(y,1)} |u_n|^2 dx
+
\sup_{y \in \RN} \int_{B(y,1)} |v_n|^2 dx
\right)
>0.
\end{equation*}
\end{claim*}
Suppose that both $\{u_n\}_{n \in \N}$ and $\{v_n\}_{n \in \N}$ vanish. 
Then we can apply the P.-L. Lions lemma \cite[Lemma I.1]{MR87e:49035b} to obtain that
$\lim_{n \to \infty} u_n = \lim_{n \to \infty} v_n = 0$ in $L^l(\RN)$, where $l=2+4/N$.
On the other hand, 
by (G1)--(G3), for any $\epsilon>0$, there exists a positive constant
 $C(G,\epsilon)$ such that
\begin{equation*}
|G(s)| \leq \epsilon (|s_1|+|s_2|) 
+ C(G,\epsilon) (|s_1|^{2/N+1} + |s_2|^{2/N +1}).
\end{equation*}
Therefore we have
\begin{equation*}
 J[u_n, v_n] \geq 
- \epsilon \int_{\RN} |u_n|^2 + |v_n|^2 dx
- C(G, \epsilon) \int_{\RN} |u_n|^l + |v_n|^l dx.
\end{equation*}
Since $\{(u_n, v_n)\}_{n \in \N}$ is the minimizing sequence over
 $M_{\alpha, \beta}$,
taking $n \to \infty$, 
we have
$E_{\alpha, \beta} \geq - \epsilon(\alpha +\beta)$.
Since $\epsilon>0$ is arbitrary, 
$E_{\alpha, \beta} \geq 0$.
It contradicts to  Lemma \ref{lem:4} (ii).

In the above claim, 
we can assume $\{u_n\}_{n \in \N}$ does not vanish without loss of generality.

\begin{claim*}
 $\{v_n\}_{n \in \N}$ does not vanish.
\end{claim*}
Suppose that $\{v_n\}_{n \in \N}$ vanish.
%\footnote{以下の議論はBrezis-Liebを使って一発では？検討の結果、$\int |G(u_n, v_n)-G(u_n,0)|$の評価なので少し違っていてダメ。}.
Since $\{v_n\}_{n \in \N}$ is bounded in $H^1(\RN)$, 
we can apply the P.-L. Lions lemma to obtain
$\lim_{n \to \infty} v_n=0$ in $L^l(\RN)$.
By using
\begin{equation*}
 G(s_1, s_2) - G(s_1, 0)
=
\int_0^1 \frac{d}{d \theta} G(s_1, \theta s_2) d\theta
=
\int_0^1 g_2(s_1, \theta s_2) s_2 d \theta
\end{equation*}
and (G1)--(G3), we have
\begin{equation*}
|G(s_1, s_2) - G(s_1,0)|
\leq
\left(\epsilon+C(G,\epsilon)(|s_1|^{2/N}+|s_2|^{2/N}) \right)|s_2|.
\end{equation*}
Thus, we can estimate as 
\begin{align*}
& \left|
\int_{\RN} G(|u_n|^2, |v_n|^2) dx 
- 
\int_{\RN} G(|u_n|^2,0) dx 
\right| \\
&\leq
\epsilon \beta+C(G,\epsilon)
\int_{\RN} (|u_n|^{4/N}+|v_n|^{4/N}) |v_n|^2 dx \\
&\leq
\epsilon \beta+C(G,\epsilon)
\left(
\|u_n\|_{L^l(\RN)}^{2l/(N+2)}
\|v_n\|_{L^l(\RN)}^{Nl/(N+2)}
+
\|v_n\|_{L^l(\RN)}^l
\right).
\end{align*}
Since
$\lim_{n \to 0} v_n=0$ in $L^l(\RN)$ and  $\epsilon>0$ is arbitrarily, 
\begin{equation*}
\int_{\RN} G(|u_n|^2, |v_n|^2) dx 
- 
\int_{\RN} G(|u_n|^2,0) dx 
= o(1)
\text{ as } n \to \infty.
\end{equation*}
Thus we obtain
\begin{equation*}
 J[u_n, v_n] \geq J[u_n, 0] +o(1) \geq E_{\alpha, 0}+o(1) \text{ as } n \to \infty.
\end{equation*}
It contradicts to the assumption (E2). Hence $\{v_n\}_{n \in \N}$ does not vanish.

Since $\{u_n\}_{n \in \N}$ and 
$\{v_n\}_{n \in \N}$ are $H^1$-bounded sequences, 
taking a subsequence, there exist
$\{y_n\}_{n \in \N} \subset \RN$,
$u \in H^1(\RN) \setminus \{0\}$, and
$v \in H^1(\RN)$ such that 
\begin{equation}
\label{eq:5}
\begin{cases}
u_n(\cdot - y_n) \rightharpoonup u, \quad
v_n(\cdot - y_n) \rightharpoonup v 
& \text{ weakly in } H^1(\RN), \\
u_n(\cdot - y_n) \to u, \quad
v_n(\cdot - y_n) \to v 
& \text{ in } L^p_{\loc}(\RN) \text{ for } p \in [1,2^*), \\
u_n(\cdot - y_n) \to u, \quad
v_n(\cdot - y_n) \to v \quad 
& \text{ a.e. in } \RN \text{ as } n \to \infty.
\end{cases}
\end{equation}

Put $\phi_n= u_n(\cdot - y_n) - u$,
$\psi_n= v_n(\cdot - y_n) - v$,
$\alpha'=\|u\|_{L^2(\RN)}^2$ and
$\beta'=\|v\|_{L^2(\RN)}^2$. 
Then 
$0 < \alpha' \leq \alpha$ and
$0 \leq \beta' \leq \beta$ hold. 
\begin{claim*}
 $\alpha'=\alpha$.
\end{claim*}
Suppose that the claim does not hold, then $\alpha'< \alpha$.
By (G1)--(G3), we can apply the Brezis-Lieb lemma \cite{MR699419} to 
obtain
\begin{equation*}
 J[u_n, v_n]=J[u, v] + J[\phi_n, \psi_n] + o(1)
\geq E_{\alpha', \beta'} + E_{\|\phi_n\|_{L^2(\RN)}^2, \|\psi_n\|_{L^2(\RN)}^2} + o(1).
\end{equation*}

Since 
$\lim_{n \to \infty} \|\phi_n\|_{L^2(\RN)}^2 = \alpha - \alpha'$ and
$\lim_{n \to \infty} \|\psi_n\|_{L^2(\RN)}^2 = \beta - \beta'$,
by Lemma \ref{lem:4} (iii),  
taking $n \to \infty$, we have
\begin{equation}
\label{eq:7}
E_{\alpha, \beta} \geq J[u,v] + E_{\alpha-\alpha', \beta-\beta'}.
\end{equation}
On the other hand, by Lemma \ref{lem:4} (i),
\begin{equation}
\label{eq:8}
J[u,v] + E_{\alpha-\alpha', \beta-\beta'} 
\geq
E_{\alpha', \beta'} +
E_{\alpha-\alpha', \beta-\beta'}
\geq E_{\alpha, \beta}.
\end{equation}
By \eqref{eq:7} and \eqref{eq:8}, we obtain that 
$(u,v)$ is a global minimizer with respect to $E_{\alpha', \beta'}$.

To obtain a contradiction, we consider two cases 
$\beta-\beta'>0$ and $\beta-\beta'=0$.
In the case $\beta-\beta'>0$, noting $\alpha-\alpha'>0$,
let $\{(\xi_n, \zeta_n)\}_{n \in \N} \subset M_{\alpha - \alpha', \beta-\beta'}$ be a minimizing sequence with respect to $E_{\alpha-\alpha', \beta-\beta'}$. Then, as discussed before, 
Neither $\{\xi_n\}_{n \in \N}$ nor $\{\zeta_n\}_{n \in \N}$ vanish. 
Therefore, taking a subsequence, there exist
$\{z_n\}_{n \in \N} \subset \RN$, 
$\xi \in H^1(\RN) \setminus \{0\}$, and
$\zeta \in H^1(\RN)$ such that
\begin{align*}
\xi_n(\cdot - z_n) & \rightharpoonup \xi, \quad
\zeta_n(\cdot - z_n) \rightharpoonup \zeta \quad 
\text{ weakly in } H^1(\RN), \\
\xi_n(\cdot - z_n) & \to \xi, \quad
\zeta_n(\cdot - z_n) \to \zeta \quad 
\text{ in } L^p_{\loc}(\RN), \\
\xi_n(\cdot - z_n) & \to \xi, \quad
\zeta_n(\cdot - z_n)  \to \zeta \quad 
\text{ a.e. in } \RN \text{ as } n \to \infty.
\end{align*}
Putting 
$\alpha''=\|\xi\|_{L^2(\RN)}^2$ and
$\beta''=\|\zeta\|_{L^2(\RN)}^2$, we have
\begin{equation*}
E_{\alpha-\alpha', \beta-\beta'} \geq E_{\alpha'', \beta''} +
 E_{\alpha-\alpha'-\alpha'', \beta-\beta'-\beta''},
\end{equation*}
and $(\xi, \zeta)$ is a global minimizer with respect to $E_{\alpha'', \beta''}$.
Hence $(\xi, \zeta)$ is a solution of
\begin{equation*}
 \Delta \xi + g_1(\xi, \zeta) = \mu \xi, \quad
 \Delta \zeta + g_2(\xi, \zeta) = \nu \zeta \text{ in } \RN,
\end{equation*}
where $\mu$ and $\nu$ is Lagrange multipliers.
By using the elliptic regularity theory, $\xi$ and $\zeta$ is of class $C^1$ and satisfy the condition (A).
Now we can apply Theorem \ref{thm:1} and Lemma \ref{lem:10} to get
\begin{equation*}
E_{\alpha'+ \alpha'', \beta'+\beta''} \leq 
 J[(u \star \xi, v \star \zeta)] < J[u,v] + J[\xi, \zeta]
=E_{\alpha', \beta'} + E_{\alpha'', \beta''}.
\end{equation*}
It contradicts to \eqref{eq:7} and \eqref{eq:8}.
In the case $\beta-\beta'=0$, we can obtain contradiction by the same argument.

Thus, we have that $\|u\|_{L^2(\RN)}^2=\alpha$ holds in \eqref{eq:5}.
On the other hand, repeating the same argument for $\{v_n\}_{n \in \N}$ instead of $\{u_n\}_{n \in \N}$,
taking a subsequence, there exist
$\{z_n\}_{n \in \RN}$, 
$\tilde{u} \in H^1(\RN)$, and
$\tilde{v} \in H^1(\RN)$ such that
\begin{equation*}
\begin{cases}
u_n(\cdot - z_n) \rightharpoonup \tilde{u}, \quad
v_n(\cdot - z_n) \rightharpoonup \tilde{v} 
& \text{ weakly in } H^1(\RN), \\
u_n(\cdot - z_n) \to \tilde{u}, \quad
v_n(\cdot - z_n) \to \tilde{v} 
& \text{ in } L^p_{\loc}(\RN), \\
u_n(\cdot - z_n) \to \tilde{u}, \quad
v_n(\cdot - z_n) \to \tilde{v} \quad 
& \text{ a.e. in } \RN \text{ as } n \to \infty.
\end{cases}
\end{equation*}
Moreover we have $\|\tilde{v}\|_{L^2(\RN)}^2 = \beta$.
\begin{claim*}
$\limsup_{n \to \infty}|y_n-z_n| < \infty$
\end{claim*}
If not, taking a subsequence, we can assume $\limsup_{n \to \infty} |y_n-z_n|= \infty$.
Since 
$\|u\|_{L^2(\RN)}^2=\alpha$ and
$\|\tilde{v}\|_{L^2(\RN)}^2=\beta$. we have
$\tilde{u}=v=0$ a.e. in $\RN$.
By the Brezis-Lieb lemma, %\footnote{詳しく}、
\begin{equation*}
 J[u_n, v_n] = J[u,0] + J[0, \tilde{v}] 
+ J[u_n-u(\cdot+y_n), v_n-\tilde{v}(\cdot + z_n)].
\end{equation*}
On the other hand, since 
$\lim_{n \to \infty} \|u_n-u(\cdot +y_n)\|_{L^2(\RN)}=
\lim_{n \to \infty} \|v_n-\tilde{v}(\cdot +y_n)\|_{L^2(\RN)}=0
$, the P.-L. Lions lemma asserts that
\begin{equation*}
 \liminf_{n \to \infty} 
J[u_n-u(\cdot+y_n), v_n-\tilde{v}(\cdot + z_n)] \geq 0.
\end{equation*}
As $n \to \infty$, we get
\begin{equation}
\label{eq:6}
 E_{\alpha, \beta} \geq J[u,0] + J[0,\tilde{v}] \geq E_{\alpha,0} + E_{0,\beta}
\geq
E_{\alpha, \beta}.
\end{equation}
It means that $(u,0)$ and $(0,\tilde{v})$ are global minimizers with respect to $E_{\alpha,0}$ and $E_{0,\beta}$.
By using (G5), %\footnote{詳しく}, 
we have 
\begin{equation*}
E_{\alpha,\beta} \leq J[u,\tilde{v}] 
<J[u,0]+J[0,\tilde{v}].
\end{equation*}
It contradicts to \eqref{eq:6}. Hence the claim holds.

Thus, taking a subsequence, there exists $z \in \RN$ such that 
$z_n = y_n + z + o(1)$ in $\RN$ as $n \to \infty$.
Put $v=\tilde{v}(\cdot + z)$ then 
\eqref{eq:5} holds for $(u,v) \in M_{\alpha,\beta}$.
For $\phi_n= u_n(\cdot-y_n)-u$ and
$\psi_n= u_n(\cdot-y_n)-v$, 
$\phi_n, \psi_n \to 0$ in $L^2(\RN)$.
By using the P.-L. Lions lemma, 
$\phi_n, \psi_n \to 0$ in $L^l(\RN)$. 
Hence $\int_{\RN} G(|\phi_n|^2, |\psi_n|^2) dx \to 0$.
By the Brezis-Lieb lemma, 
\begin{align*}
 J[u_n, v_n] 
&= J[u,v] + J[\phi_n, \psi_n] + o(1) \\
&= E_{\alpha, \beta} 
+ \frac{1}{2} \int_{\RN} |\nabla \phi_n|^2 + |\nabla \phi_n|^2 dx
+o(1) \text{ as } n \to \infty.
\end{align*}
Taking $n \to \infty$, we obtain
\begin{equation*}
\lim_{n \to \infty} \int_{\RN} |\nabla \phi_n|^2 + |\nabla \psi_n|^2 dx =0.
\end{equation*}
Thus we get 
$\lim_{n \to \infty}\phi_n=
\lim_{n \to \infty} \psi_n = 0$ in $H^1(\RN)$. It means the conclusion.
\end{proof}

\appendix
\section{Appendix}
\label{sec:a}

In this section,  we give the proofs of lemmas used in the above section.
\begin{lemma}
\label{lem:9}
Assume (G1)--(G4).
For $R>0$, there exists a constant $C(N,G,R) >0$ such that 
\begin{equation}
\label{eq:20}
\frac{1}{4} \left(\|\nabla u\|_{L^2(\RN)}^2 + \|\nabla v\|_{L^2(\RN)}^2 \right)
\leq 
J[u, v] + C(N, G, R)
\end{equation}
for $(u, v) \in M_{\alpha, \beta}$ with  $\alpha, \beta \in [0,R]$.
Moreover, for $\alpha, \beta \geq 0$, 
any minimizing sequence $\{(u_n, v_n)\}_{n \in \N} \subset M_{\alpha,\beta}$ is $H^1$-bounded.
\end{lemma}

\begin{proof} 
By (G1)--(G3),
for any $\epsilon>0$, there exists $C(G,\epsilon)>0$ such that
\begin{equation*}
 |G(s_1, s_2)| \leq C(G,\epsilon) (|s_1|+|s_2|) + 
\epsilon (|s_1|^{2/N+1} + |s_2|^{2/N+1}).
\end{equation*}
Therefore, by using the Gagliardo-Nirenberg inequality, 
for $(u,v) \in M_{\alpha,\beta}$, 
we have
\begin{align*}
 J[u, v] 
\geq & 
-C(G,\epsilon) (\alpha+ \beta)
+ \frac{1}{2} \left(\|\nabla u\|_{L^2(\RN)}^2 + \|\nabla v\|_{L^2(\RN)}^2 \right) \\
&- \epsilon \left(\|u\|_{L^l(\RN)}^l + \|v\|_{L^l(\RN)}^l \right) \\
\geq & 
-C(G,\epsilon) (\alpha+ \beta)
+ \frac{1}{2} \left(\|\nabla u\|_{L^2(\RN)}^2 + \|\nabla v\|_{L^2(\RN)}^2 \right) \\
& - \epsilon C(N) \left(\alpha^{4/N} \|\nabla u\|_{L^2(\RN)}^2 
+ \beta^{4/N}\|\nabla v\|_{L^2(\RN)}^2 \right) \\
 \geq&
-2R C(G,\epsilon) 
+ \left(\frac{1}{2} - \epsilon C(N) R^{4/N}\right) \left(\|\nabla u\|_{L^2(\RN)}^2 + \|\nabla v\|_{L^2(\RN)}^2 \right) 
\end{align*}
Choosing $\epsilon >0$ satisfying $\epsilon C(N) R^{4/N} < 1/4$,
we have \eqref{eq:20}.

Let $\{(u_n, v_n)\}_{n \in \N} \subset M_{\alpha, \beta}$ be a minimizing sequence.
Since $\{(u_n, v_n)\}_{n \in \N} \subset M_{\alpha,\beta}$,
$\{u_n\}_{n \in \N}$ and 
$\{v_n\}_{n \in \N}$ are bounded in $L^2(\RN)$.
\eqref{eq:20} asserts that $H^1$-boundedness.

\end{proof}

\begin{proof}[Proof of Lemma \ref{lem:4}]
(i): For $\epsilon >0$, there exists $(u,v) \in M_{\alpha,\beta} \cap C_0^\infty(\RN)$ and 
$(\phi,\psi) \in M_{\alpha',\beta'} \cap C_0^\infty(\RN)$. By using parallel transformation, we can assume 
that $(\supp u \cup \supp v) \cap (\supp \phi \cup \supp \psi) = \emptyset$.
Therefore $(u+\phi, v+\psi) \in M_{\alpha+\alpha', \beta+\beta'}$ and 
\begin{equation*}
E_{\alpha+\alpha', \beta+\beta'}
\leq 
I[u+\phi, v+\psi]
=
I[u, v]
+
I[\phi, \psi]
\leq
E_{\alpha,\beta} + E_{\alpha', \beta'} + 2 \epsilon.
\end{equation*}
Since $\epsilon>0$ is arbitrarily, it asserts (i).

(ii): (i) and (E2) asserts (ii) immediately.

(iii): First we show the following.
\begin{claim}
 For $\alpha ,\beta>0$, $\liminf_{(h,k) \to 0} E_{\alpha + h, \beta + k} \geq E_{\alpha, \beta}$.
\end{claim}
Put $R=\max \{ \alpha +1, \beta +1\}$ and assume $|h|, |k| <\min\{\alpha ,\beta, 1\}$. We note that $0 < \alpha + h \leq R$ and $0 < \beta+k \leq R$.
For $\epsilon >0$, by the definition of $E_{\alpha+h, \beta+k}$,
there exists $(u,v) \in M_{\alpha+h, \beta+k}$ such that
\begin{equation*}
 E_{\alpha + h, \beta + k} \leq J[u,v] \leq  E_{\alpha + h, \beta + k} +\epsilon.
\end{equation*}
Putting
\begin{equation*}
 t=t(h,k) = \left(
\min \left\{
\frac{\alpha}{\alpha+h}, 
\frac{\beta}{\beta+k} 
\right\}
\right)^{1/N},
\end{equation*}
$u_t(x)=u(x/t)$, and $v_t(x)=v(x/t)$,
we have 
\begin{equation}
\label{eq:21}
 \lim_{(h,k) \to (0,0)} t =1,
\end{equation}
$ \|u_t\|_{L^2(\RN)}^2 = t^N (\alpha +h) \leq \alpha$, and 
$\|v_t\|_{L^2(\RN)}^2 = t^N (\beta +k)  \leq \beta$.
Therefore, by using (i) and (ii), we obtain
\begin{equation*}
 J[u_t, v_t] \geq E_{t^N(\alpha + h), t^N(\beta+k)} \geq
E_{\alpha, \beta}.
\end{equation*}
On the other hand, 
\begin{align*}
 J[u_t, v_t] &= \frac{t^{N-2}}{2} \int_{\RN} |\nabla u|^2 + |\nabla v|^2 dx
- t^N \int_{\RN} G(|u|^2, |v|^2) dx \\
& \leq 
 t^N J[u,v] + \frac{t^{N-2}\left|1-t^2\right|}{2} \int_{\RN} |\nabla u|^2 + |\nabla v|^2 dx.
\end{align*}
By Lemma \ref{lem:9},
\begin{align*}
 \int_{\RN} |\nabla u|^2 + |\nabla v|^2 dx & \leq J[u,v] + C(N,G,R) \\
& \leq \epsilon + C(N,G,R).
\end{align*}
Thus, noting \eqref{eq:21}, we get
\begin{equation*}
E_{\alpha, \beta} \leq \liminf_{(h,k) \to (0,0)} E_{\alpha+h,\beta+k} + \epsilon.
\end{equation*}
Since we can take $\epsilon >0$ arbitrarily, the claim holds.

\begin{claim}
 For $\alpha ,\beta>0$, $\limsup_{(h,k) \to 0} E_{\alpha + h, \beta + k} \leq E_{\alpha, \beta}$.
\end{claim}
We can show the claim as before. 
Actually,
for $\epsilon >0$, 
there exists $(u,v) \in M_{\alpha, \beta}$ such that
\begin{equation*}
 E_{\alpha, \beta } \leq J[u,v] \leq  E_{\alpha, \beta} +\epsilon.
\end{equation*}
Putting
\begin{equation*}
 t=t(h,k) = \left(
\min \left\{
\frac{\alpha+h}{\alpha}, 
\frac{\beta+k}{\beta} 
\right\}
\right)^{1/N},
\end{equation*}
$u_t(x)=u(x/t)$, and $v_t(x)=v(x/t)$,
we have $\lim_{(h,k) \to (0,0)} t =1$,
$ \|u_t\|_{L^2(\RN)}^2 = t^N \alpha \leq \alpha+h$, and 
$\|v_t\|_{L^2(\RN)}^2 = t^N \beta \leq \beta+k$.
Therefore, we obtain
\begin{equation*}
 J[u_t, v_t] \geq E_{t^N \alpha, t^N \beta} \geq
E_{\alpha+h, \beta+k}.
\end{equation*}
On the other hand, 
\begin{align*}
 J[u_t, v_t] &= \frac{t^{N-2}}{2} \int_{\RN} |\nabla u|^2 + |\nabla v|^2 dx
- t^N \int_{\RN} G(|u|^2, |v|^2) dx \\
& \leq 
 t^N J[u,v] + \frac{t^{N-2}\left|1-t^2\right|}{2} \int_{\RN} |\nabla u|^2 + |\nabla v|^2 dx.
\end{align*}
Since $u$ and $v$ are independent of $h$ and $k$, by \eqref{eq:21}, we get
\begin{equation*}
\limsup_{(h,k) \to (0,0)} E_{\alpha+h, \beta+k} \leq E_{\alpha, \beta} + \epsilon.
\end{equation*}
Since we can take $\epsilon >0$ arbitrarily, the claim holds.

Next, we consider the case $\alpha=0$ or $\beta=0$.
It is sufficient to consider the case $\beta=0$. 
By the same argument as above, we can show $\alpha \mapsto E_{\alpha, 0}$ is continuous. 
Therefore, we show the following claim.
\begin{claim}
$\lim_{k \to 0} E_{\alpha ,k}= E_{\alpha, 0}$ uniformly with respect to  $\alpha \in [0,R]$.
\end{claim}
For $\epsilon>0$, there exists $(u,v) \in M_{\alpha,k}$ such that 
\begin{equation*}
 J[u,v] \leq E_{\alpha ,k} + \epsilon.
\end{equation*}
On the other hand, we have
\begin{align*}
J[u,v] & \geq J[u,0] + \int_{\RN} G(|u^2|, 0) - G(|u|^2, |v|^2) dx \\
& \geq E_{\alpha,0} + \int_{\RN} G(|u^2|, 0) - G(|u|^2, |v|^2) dx.
\end{align*}
By using
\begin{equation*}
 G(s_1, s_2) - G(s_1, 0)
=
\int_0^1 \frac{d}{d \theta} G(s_1, \theta s_2) d\theta
=
\int_0^1 g_2(s_1, \theta s_2) s_2 d \theta
\end{equation*}
and (G1)--(G3), we have
\begin{equation*}
|G(s_1, s_2) - G(s_1,0)|
\leq
\left(C(G,\delta) + \delta(|s_1|^{2/N}+|s_2|^{2/N}) \right)|s_2|.
\end{equation*}
Thus, 
\begin{align*}
& \left|
\int_{\RN} G(|u|^2, |v|^2) dx 
- 
\int_{\RN} G(|u|^2,0) dx 
\right| \\
&\leq
k C(G, \delta) + \delta
\int_{\RN} (|u|^{4/N}+|v|^{4/N}) |v|^2 dx \\
&\leq
kC(G,\delta) + \delta
\left(
\|u\|_{L^l(\RN)}^{2l/(N+2)}
\|v\|_{L^l(\RN)}^{Nl/(N+2)}
+
\|v\|_{L^l(\RN)}^l
\right) \\
& \leq
kC(G,\delta) + \delta C(N)
\left(
\|\nabla u\|_{L^2(\RN)}^{2l/(N+2)}
\|\nabla v\|_{L^2(\RN)}^{Nl/(N+2)}
+
\|\nabla v\|_{L^2(\RN)}^l
\right) 
\end{align*}
By Lemma \ref{lem:9}, 
\begin{equation*}
\|\nabla u\|_{L^2(\RN)}^2 + 
\|\nabla v\|_{L^2(\RN)}^2
\leq 4 (E_{\alpha,k} +\epsilon) + C(N,G,R)
\leq C(N,G,R)
\end{equation*}
for $\epsilon \leq 1$, because of $E_{\alpha, k} \leq 0$.
Consequently we have
\begin{align*}
\limsup_{k \to 0} \left|
\int_{\RN} G(|u|^2, |v|^2) dx 
- 
\int_{\RN} G(|u|^2,0) dx 
\right|
& \leq
\limsup_{k \to 0} 
\left(
k C(G, \delta) + \delta C(N,G,R)
\right) \\
& \leq 
\delta C(N,G,R).
\end{align*}
Since $\delta >0$ is arbitrarily, 
\begin{equation*}
\lim_{k \to 0} \left|
\int_{\RN} G(|u|^2, |v|^2) dx 
- 
\int_{\RN} G(|u|^2,0) dx 
\right|=0 \text{ uniformly with respect to } \alpha.
\end{equation*}
Thus we have 
\begin{equation*}
 E_{\alpha, 0} \leq \liminf_{k \to 0} E_{\alpha,k}+\epsilon 
\text{ uniformly with respect to } \alpha.
\end{equation*}
Since $\epsilon >0$ is arbitrarily, 
\begin{equation*}
 E_{\alpha, 0} \leq \liminf_{k \to 0} E_{\alpha,k}
\text{ uniformly with respect to } \alpha.
\end{equation*}
On the other hand, by (i) and (ii), $E_{\alpha, k} \leq E_{\alpha,0}$ holds.
Thus we get the conclusion.
\end{proof}

%\begin{lemma}\footnote{Brezis-Lieb}
%\label{lem:8}
%Suppose (G1)--(G4).
%If $\{u_n\}_{n \in \N} \subset H^1(\RN)$ and $\{v_n\}_{n \in \N} \subset H^1(\RN)$ satisfy
%\begin{align*}
% & u_n \rightharpoonup u \text{ weakly in } H^1(\RN), \\
% & v_n \rightharpoonup v \text{ weakly in } H^1(\RN), \\
% & u_n \to u \text{ a.e. in } \RN, \\
% & v_n \to v \text{ a.e. in } \RN \text{ as } n \to \infty. 
%\end{align*}
%Then 
%\begin{equation*}
%\lim_{n \to \infty} \int_{\RN} |G(u_n, v_n) - G(u,v) - G(u_n-u, v_n-v)| dx.
%\end{equation*}
%\end{lemma}
%\begin{proof}
%\footnote{証明をつける}
%\end{proof}

\begin{lemma}
\label{lem:10}
Assume (G1)--(G4). For $u, v, \phi, \psi \in H^1(\RN)$ satisfying the condition (A),
\begin{equation*}
\int_{\RN} G((u \star \phi)^2, (v \star \psi)^2) dx 
\geq
\int_{\RN} G(|u|^2, |v|^2) dx 
+
\int_{\RN} G(|\phi|^2, |\psi|^2) dx 
\end{equation*}
\end{lemma}

\begin{proof}
For simplicity, we use $u,v, \phi, \psi$ instead of $|u|^2$, $|v|^2$, $|\phi|^2$, $|\psi|^2$.
Noting $(u \star \phi)^2 =|u|^2 \star |\phi|^2$ and 
$(v \star \psi)^2 =|v|^2 \star |\psi|^2$, we show 
\begin{equation}
\label{eq:22}
\int_{\RN} G(u \star \phi, v \star \psi) dx 
\geq
\int_{\RN} G(u, v) dx 
+
\int_{\RN} G(\phi, \psi) dx 
\end{equation}
for $u,v, \psi, \psi \geq 0$.

By (G2) and (G4), 
$g_j(s,t) \geq g_j(0,0) =0$.
By using mean value theorem, we have
\begin{align*}
& \int_\RN G(u,v) dx \\
= &
\int_\RN G(u(x), v(x)) - G(0,v(x)) + G(0,v(x)) - G(0,0) dx \\
=&
\int_\RN dx \int_0^{u(x)} g_2(s, v(x)) ds + 
\int_\RN dx \int_0^{v(x)} g_1(0, t) dt  \\
=&
\int_\RN dx \int_0^\infty g_2(s,v(x)) \chi_{\{x;u(x)>s\}}(x)ds
+
\int_\RN dx \int_0^\infty g_1(0,t) \chi_{\{x;v(x)>t\}}(x)dt \\
=&
\int_\RN dx \int_0^\infty \int_0^\infty 
\chi_{\{x; g_2(s,v(x))>r\}}(x) \chi_{\{x;u(x)>s\}}(x)dr ds \\
&+
\int_\RN dx \int_0^\infty g_1(0,t) \chi_{\{x;v(x)>t\}}(x)dt \\
=&
\int_0^\infty \int_0^\infty 
|\{x; g_2(s,v(x))>r\} \cap  \{x;u(x)>s\}| dr ds  +
\int_0^\infty g_1(0,t) |\{x;v(x)>t\}| dt.
\end{align*}
For each $r,s>0$, Put $t(r,s) =\sup \{t; g_2(s,t) \leq r\}$ if 
$\{t; g_2(s,t) \leq r\}\not=\emptyset$
, $t(r,s)=-\infty$ if $\{t; g_2(s,t) \leq r\}=\emptyset$. 
Then, by (G4), $g_2(s,v(x))>r$ if and only if $v(x)>t(r,s)$.
So we have
\begin{equation*}
|\{x; g_2(s,v(x))>r\} \cap  \{x;u(x)>s\}|
=|\{x; v(x)>t(r,s)\} \cap  \{x;u(x)>s\}|.
\end{equation*}
Hence
\begin{align}
\notag
 \int_\RN G(u,v) dx
=&
\int_0^\infty \int_0^\infty 
|\{x; v(x)>t(r,s)\} \cap  \{x;u(x)>s\}| dr ds \\
\label{eq:40}
& +
\int_0^\infty g_1(0,t) |\{x;v(x)>t\}| dt.
\end{align}
Similarly, we can obtain
\begin{align}
\nonumber
 \int_\RN G(\phi,\psi) dx
=& 
\int_0^\infty \int_0^\infty 
|\{x; \psi(x)>t(r,s)\} \cap  \{x;\phi(x)>s\}| dr ds \\
\label{eq:41}
& + 
\int_0^\infty g_1(0,t) |\{x;\psi(x)>t\}| dt, \\
\nonumber
 \int_\RN G(u \star \phi,v \star \psi) dx
=& 
\int_0^\infty \int_0^\infty 
|\{x; (v \star \psi) (x)>t(r,s)\} \cap  \{x;(u \star \phi)(x)>s\}| dr ds \\
\label{eq:42}
& +
\int_0^\infty g_1(0,t) |\{x;(v \star \psi)(x)>t\}| dt.
\end{align}
Here, by Lemma \ref{lem:3} (i), we have
\begin{align*}
 & 
|\{x; v(x)>t(r,s)\} \cap  \{x;u(x)>s\}|
+|\{x; \psi(x)>t(r,s)\} \cap  \{x;\phi(x)>s\}| \\
 \leq &
\min\{|\{x; v(x)>t(r,s)\}|, |\{x;u(x)>s\}|\}
+ \min\{|\{x; \psi(x)>t(r,s)\}|, |\{x; \phi(x)>s\}|\} \\
\leq &
\min\{|\{x; v(x)>t(r,s)\}|+|\{x; \psi(x)>t(r,s)\}|, |\{x;u(x)>s\}|+|\{x;
 \phi(x)>s\}|\} \\
= &
\min\{|\{x; (v \star \psi)(x)>t(r,s)\}|, |\{x;(u \star \phi)(x)>s\}|\}.
\end{align*}
Since $\{x; (v \star \psi)(x)>t(r,s)\}$, $\{x;(u \star \phi)(x)>s\}$ are
 balls centered at the origin, we have 
\begin{align*}
& \min\{|\{x; (v \star \psi)(x)>t(r,s)\}|, |\{x;(u \star \phi)(x)>s\}|\} \\
&= |\{x; (v \star \psi)(x)>t(r,s)\} \cap \{x;(u \star \phi)(x)>s\}|.
\end{align*}
Hence, 
\begin{align}
\notag
& |\{x; v(x)>t(r,s)\} \cap  \{x;u(x)>s\}|
+|\{x; \psi(x)>t(r,s)\} \cap  \{x;\phi(x)>s\}| \\
\label{eq:39}
& \leq 
|\{x; (v \star \psi)(x)>t(r,s)\} \cap \{x;(u \star \phi)(x)>s\}|. 
\end{align}
On the other hand, 
\begin{equation}
\label{eq:56}
|\{x;v(x)>t\}|
+ |\{x;\psi(x)>t\}|
=
|\{x;(v \star \psi)(x)>t\}|
\end{equation}
because of Lemma \ref{lem:3} (i).
Consequently, \eqref{eq:40}, \eqref{eq:41}, \eqref{eq:42},
 \eqref{eq:39} and \eqref{eq:56} assert that this lemma.
\end{proof}

\bibliographystyle{plain}

%\bibliography{mathscinet,preprint}

\end{document}